\documentclass[runningheads]{llncs}
\usepackage{graphicx}
\usepackage{longtable}
\usepackage{wrapfig}
\usepackage{ucs}
\usepackage{dsfont}
\usepackage{mathrsfs}
\usepackage{mathtools}
\usepackage{amsmath}
\usepackage{amsfonts}
\usepackage{amssymb}
\usepackage{soul}
\usepackage[makeroom]{cancel}
\usepackage{bbm}
\usepackage[english]{babel}
\usepackage{ucs}
\usepackage[colorlinks=true,linkcolor=blue,citecolor=blue,urlcolor=orange]{hyperref}

\usepackage{longtable}
\usepackage{wasysym}
\usepackage{verbatim}
\usepackage{multirow}
\usepackage{bussproofs}
\usepackage{setspace}
\usepackage{graphicx}
\usepackage{stmaryrd}
\usepackage{forest}
\forestset{smullyan tableaux/.style={for tree={math content},where n children=1{!1.before computing xy={l=\baselineskip},!1.no edge}{},closed/.style={label=below:$\times$},},}
\usepackage{longtable}
\usepackage[all]{xy}
\usepackage{wrapfig}
\usepackage{xcolor}
\usepackage[colorinlistoftodos,prependcaption,textsize=small]{todonotes}
\usepackage[most]{tcolorbox}
\usepackage{enumitem}
\usepackage{tikz}
\usetikzlibrary{arrows,calc,patterns,positioning,shapes}
\usetikzlibrary{decorations.pathmorphing}
\tikzset{
modal/.style={>=stealth',shorten >=1pt,shorten <=1pt,auto,
node distance=1.5cm,semithick},
world/.style={circle,draw,minimum size=1cm},
point/.style={circle,draw,fill=black,inner sep=0.5mm},
reflexive/.style={->,in=120,out=60,loop,looseness=#1},
reflexive/.default={5},
reflexive point/.style={->,in=135,out=45,loop,looseness=#1},
reflexive point/.default={25},
}
\newcommand{\commentSabine}[1]{}

\newcommand{\BD}{\mathsf{BD}}
\newcommand{\fourLukProb}{\four\Prob^{\Luk_\triangle}}
\newcommand{\HfourLukProb}{\mathcal{H}\fourLukProb}
\newcommand{\LfourLukProb}{\mathcal{L}_{\fourLukProb}}
\newcommand{\LBD}{\mathcal{L}_\mathsf{BD}}
\newcommand{\Luk}{{\mathchoice{\mbox{\rm\L}}{\mbox{\rm\L}}{\mbox{\rm\scriptsize\L}}{\mbox{\rm\tiny\L}}}}
\newcommand{\LLuk}{\mathcal{L}_{\Luk}}
\newcommand{\LLuksquare}{\mathcal{L}_{\Luk^2_\triangle}}
\newcommand{\LLukProbsquare}{\mathcal{L}_{\Prob^{\Luk^2}_\triangle}}
\newcommand{\LukProbsquare}{\Prob^{\Luk^2}_\triangle}
\newcommand{\Prop}{\mathtt{Prop}}

\newcommand{\np}{\mathsf{NP}}
\newcommand{\conp}{\mathsf{coNP}}
\newcommand{\four}{\mathbf{4}}

\newcommand{\true}{\mathbf{T}}
\newcommand{\both}{\mathbf{B}}
\newcommand{\neither}{\mathbf{N}}
\newcommand{\false}{\mathbf{F}}

\newcommand{\purebel}{\mathsf{b}}
\newcommand{\puredisbel}{\mathsf{d}}
\newcommand{\confl}{\mathsf{c}}
\newcommand{\uncert}{\mathsf{u}}
\newcommand{\purebelmod}{\mathsf{Bl}}
\newcommand{\puredisbelmod}{\mathsf{Db}}
\newcommand{\conflmod}{\mathsf{Cf}}
\newcommand{\uncertmod}{\mathsf{Uc}}

\newcommand{\Prob}{\mathsf{Pr}}
\newcommand{\Lit}{\mathsf{Lit}}
\newcommand{\Sf}{\mathtt{Sf}}

\newtheorem{convention}{Convention}

\usepackage{comment}

\begin{document}

\setlength{\jot}{0pt} 
\setlength{\abovedisplayskip}{2pt}
\setlength{\belowdisplayskip}{2pt}
\setlength{\abovedisplayshortskip}{1pt}
\setlength{\belowdisplayshortskip}{1pt}

\title{Two-layered logics for paraconsistent probabilities\thanks{The research of Marta B\'ilkov\'a was supported by the grant 22-01137S of the Czech Science Foundation. The research of Sabine Frittella and Daniil Kozhemiachenko was funded by the grant ANR JCJC 2019, project PRELAP (ANR-19-CE48-0006). This research is part of the MOSAIC project financed by the European Union's Marie Sk\l{}odowska-Curie grant No.~101007627.}}
\titlerunning{Paraconsistent non-standard modalities}
\author{Marta B\'ilkov\'a\inst{1}\orcidID{0000-0002-3490-2083} \and Sabine Frittella\inst{2}\orcidID{0000-0003-4736-8614}\and Daniil Kozhemiachenko\inst{2}\orcidID{0000-0002-1533-8034} \and \\Ondrej Majer\inst{3}\orcidID{0000-0002-7243-1622}}
\authorrunning{B\'ilkov\'a et al.}
\institute{The Czech Academy of Sciences, Institute of Computer Science, Prague\\
\email{bilkova@cs.cas.cz}
\and
INSA Centre Val de Loire, Univ.\ Orl\'{e}ans, LIFO EA 4022, France\\
\email{sabine.frittella@insa-cvl.fr, daniil.kozhemiachenko@insa-cvl.fr}
\and
The Czech Academy of Sciences, Institute of Philosophy, Prague\\\email{majer@flu.cas.cz}}
\maketitle              
\begin{abstract}
We discuss two two-layered logics formalising reasoning with paraconsistent probabilities that combine the \L{}ukasiewicz $[0,1]$-valued logic with Baaz $\triangle$ operator and the Belnap--Dunn logic. The first logic $\LukProbsquare$ (introduced in~\cite{BilkovaFrittellaKozhemiachenkoMajerNazari2022arxiv}) formalises a ‘two-valued’ approach where each event $\phi$ has independent positive and negative measures that stand for, respectively, the likelihoods of $\phi$ and $\neg\phi$. The second logic $\fourLukProb$ that we introduce here corresponds to ‘four-valued’ probabilities. There, $\phi$ is equipped with four measures standing for pure belief, pure disbelief, conflict and uncertainty of an agent in $\phi$.

We construct faithful embeddings of $\fourLukProb$ and $\LukProbsquare$ into one another and axiomatise $\fourLukProb$ using a Hilber-style calculus. We also establish the decidability of both logics and provide complexity evaluations for them using an expansion of the constraint tableaux calculus for $\Luk$.
\end{abstract}
\section{Introduction\label{sec:introduction}}
Classical probability theory studies probability measures: maps from a probability space to $[0,1]$ that satisfy the (finite or countable) additivity condition:
\begin{align*}
\mu\left(\bigcup\limits_{i\in I}E_i\right)&=\sum\limits_{i\in I}\mu(E_i)\tag{$\forall i,j\in I:i\neq j\Rightarrow E_i\cap E_j=\varnothing$}
\end{align*}
Above, the disjointness of $E_i$ and $E_j$ can be construed as their incompatibility. Most importantly, if a~propositional formula $\phi$ is associated with an event (and interpreted as a statement about it), then $\phi$ and $\neg\phi$ are incompatible and $\phi\vee\neg\phi$ exhausts the entire sample space.

\emph{Paraconsistent} probability theory, on the other hand, assumes that the probability measure of an event represents not the likelihood of it happening but an agent's certainty therein which they infer from the information given by the sources. As a~\emph{single} source can give incomplete or contradictory information, it is reasonable to assume that a ‘contradictory’ event $\phi\wedge\neg\phi$ can have a positive probability and that $\phi\vee\neg\phi$ does not necessarily have probability $1$.

Thus, a~logic describing events should allow them to be both true and false (if the source gives contradictory information) or neither true nor false (when the source does not give information). Formally, this means that $\neg$ does not correspond to the complement in the sample space.

\vspace{.5em}

\textbf{Paraconsistent probabilities in $\BD$}
The simplest logic to represent reasoning about information provided by sources is the Belnap--Dunn logic~\cite{Dunn1976,Belnap2019,Belnap1977fourvalued}. Originally, $\BD$ was presented as a four-valued propositional logic in the $\{\neg,\wedge,\vee\}$ language. The values represent the different accounts a source can give regarding a~statement~$\phi$:
\begin{itemize}[noitemsep,topsep=2pt]
\item $\mathbf{T}$ stands for ‘the source only says that $\phi$ is true’;
\item $\mathbf{F}$ stands for ‘the source only says that $\phi$ is false’;
\item $\mathbf{B}$ stands for ‘the source says both that $\phi$ is false and that $\phi$ is true’;
\item $\mathbf{N}$ stands for ‘the source does not say that $\phi$ is false nor that it is true’.
\end{itemize}
The interpretation of the truth values allows for a reformulation of $\BD$ semantics in terms of \emph{two classical but independent valuations}. Namely,
\begin{center}
\begin{tabular}{c|c|c}
&\textbf{is true when}&\textbf{is false when}\\\hline
$\neg\phi$&$\phi$ is false&$\phi$ is true\\
$\phi_1\wedge\phi_2$&$\phi_1$ and $\phi_2$ are true&$\phi_1$ is false or $\phi_2$ is false\\
$\phi_1\vee\phi_2$&$\phi_1$ is true or $\phi_2$ is true&$\phi_1$ and $\phi_2$ are false
\end{tabular}
\end{center}
It is easy to see that there are no universally true nor universally false formulas in $\BD$. Thus, $\BD$ satisfies the desiderata outlined above.

The first representation of paraconsistent probabilities in terms of $\BD$ was given in~\cite{Dunn2010}, however, no axiomatisation was provided. Dunn proposes to divide the sample space into four exhaustive and mutually exclusive parts depending on the Belnapian value of $\phi$. An alternative approach was proposed in~\cite{KleinMajerRad2021}. There, the authors propose two equivalent interpretations based on the two formulations of semantics. The first option is to give $\phi$ two \emph{independent probability measures}: the one determining the likelihood of $\phi$ to be true and the other the likelihood of $\phi$ to be false. The second option follows Dunn and also divides the sample space according to whether $\phi$ has value $\true$, $\both$, $\neither$, or $\false$ in a given state.

The main difference between these two approaches is that in~\cite{Dunn2010}, the probability of $\phi\wedge\phi'$ is entirely determined by those of $\phi$ and $\phi'$ which makes it compositional. On the other hand, the paraconsistent probabilities proposed in~\cite{KleinMajerRad2021} are not compositional w.r.t.\ conjunction. In this paper, we choose the latter approach since it can be argued~\cite{Dubois2008} that belief is not compositional.

A similar approach to paraconsistent probabilities can be found in, e.g.~\cite{Bueno-SolerCarnielli2016,RodriguesBueno-SolerCarnielli2021}. There, probabilities are defined over an extension of $\BD$ with classicality and non-classicality operators. It is worth mentioning that the proposed axioms of probability are very close to those from~\cite{KleinMajerRad2021}: e.g., both allow measures $\mathtt{p}$ s.t.\ $\mathtt{p}(\phi)+\mathtt{p}(\neg\phi)<1$ (if the information regarding $\phi$ is incomplete) or $\mathtt{p}(\phi)+\mathtt{p}(\neg\phi)>1$ (when the information is contradictory).

\vspace{.5em}

\textbf{Two-layered logics for uncertainty}
Reasoning about uncertainty can be formalised via modal logics where the modality is interpreted as a measure of an event. The concrete semantics of the modality can be defined in two ways. First, using a modal language with Kripke semantics where the measure is defined on the set of states as done in, e.g.,~\cite{Gardenfors1975,DelgrandeRenne2015,DelgrandeRenneSack2019} for qualitative probabilities and in~\cite{Ilic-StepicKnezevicOgnjanovic2022} for the quantitative ones. Second, employing a two-layered formalism (cf.~\cite{FaginHalpernMegiddo1990,FaginHalpern1991ComputationalIntelligence}, \cite{BaldiCintulaNoguera2020}, and~\cite{BilkovaFrittellaKozhemiachenkoMajerNazari2022arxiv,BilkovaFrittellaKozhemiachenkoMajer2023IJAR} for examples). There, the logic is split into two levels: the inner layer describes events, and the outer layer describes the reasoning with the measure defined on events. The measure is a \emph{non-nesting} modality $\mathtt{M}$, and the outer-layer formulas are built from ‘modal atoms’ of the form $\mathtt{M}\phi$ with $\phi$ being an inner-layer formula. The outer-layer formulas are then equipped with the semantics of a~fuzzy logic that permits necessary operations (e.g., \L{}ukasiewicz for the quantitative reasoning and G\"{o}del for the qualitative).

In this work, we choose the two-layered approach. First, it is more modular than the usual Kripke semantics: as long as the logic of the event description is chosen, we can define different measures on top of it using different upper-layer logics. Second, the completeness proof is very simple since one only needs to translate the axioms of the given measure into the outer-layer logic. Finally, even though, the traditional Kripke semantics is more expressive than two-layered logics, this expressivity is not really necessary in many contexts. Indeed, people rarely say something like ‘it is probable that it is probable that $\phi$’. Moreover, it is considerably more difficult to motivate the assignment of truth values in the nesting case, in particular, when one and the same measure is applied both to a~propositional and modalised formula as in, e.g., $\mathtt{M}(p\wedge\mathtt{M}q)$.

\vspace{.5em}

\textbf{Plan of the paper}
Our paper continues the project proposed in~\cite{BilkovaFrittellaMajerNazari2020} and continued in~\cite{BilkovaFrittellaKozhemiachenkoMajerNazari2022arxiv} and~\cite{BilkovaFrittellaKozhemiachenkoMajer2023IJAR}. Here, we set to provide a logic that formalises the reasoning with four-valued probabilities as presented in~\cite{KleinMajerRad2021}. The rest of the text is organised as follows. In Section~\ref{sec:probabilities}, we recall two approaches to probabilities over $\BD$ from~\cite{KleinMajerRad2021}. In Section~\ref{sec:logics}, we provide the semantics of our two-layered logics and in Section~\ref{sec:axioms}, we axiomatise them using Hilbert-style calculi. In Section~\ref{sec:complexity}, we prove that all our logics are decidable and establish their complexity evaluations. Finally, we wrap up our results in Section~\ref{sec:conclusion}.
\section{Two approaches to paraconsistent probabilities\label{sec:probabilities}}
We begin with defining the semantics of $\BD$ on sets of states.
The language of $\BD$ is given by the following grammar (with $\Prop$ being a countable set of propositional variables).
\[\LBD\ni\phi\coloneqq p\in\Prop\mid\neg\phi\mid(\phi\wedge\phi)\mid(\phi\vee\phi)\]
\begin{convention}
In what follows, we will write $\Prop(\phi)$ to denote the set of variables occurring in $\phi$ and $\Lit(\phi)$ to denote the set of \emph{literals} (i.e., variables or their negations) occurring in $\phi$. Moreover, we use $\Sf(\phi)$ to stand for the set of all subformulas of $\phi$.

We are also going to use two kinds of formulas: the single- and the two-layered ones. To make the differentiation between them simpler, we use Greek letters from the end of the alphabet ($\phi$, $\chi$, $\psi$, etc.) to designate the first kind and the letters from the beginning of the alphabet ($\alpha$, $\beta$, $\gamma$, \ldots) for the second kind.

Furthermore, we use $v$ (with indices) to stand for the valuations of single-layered formulas and $e$ (with indices) for the two-layered formulas.
\end{convention}
\begin{definition}[Set semantics of $\BD$]\label{def:BDframesemantics}
Let $\phi,\phi'\in\LBD$, $W\neq\varnothing$, and $v^+,v^-:\Prop\to2^W$. For a~model $\mathfrak{M}=\langle W,v^+,v^-\rangle$, we define notions of $w\vDash^+\phi$ and $w\vDash^-\phi$ for $w\in W$ as follows.
\begin{align*}
w\vDash^+p&\text{ iff }w\in v^+(p)&w\vDash^-p&\text{ iff }w\in v^-(p)\\
w\vDash^+\neg\phi&\text{ iff }w\vDash^-\phi&w\vDash^-\neg\phi&\text{ iff }w\vDash^+\phi\\
w\vDash^+\phi\wedge\phi'&\text{ iff }w\vDash^+\phi\text{ and }w\vDash^+\phi'&w\vDash^-\phi\wedge\phi'&\text{ iff }w\vDash^-\phi\text{ or }w\vDash^-\phi'\\
w\vDash^+\phi\vee\phi'&\text{ iff }w\vDash^+\phi\text{ or }w\vDash^+\phi'&w\vDash^-\phi\vee\phi'&\text{ iff }w\vDash^-\phi\text{ and }w\vDash^-\phi'
\end{align*}
We denote the positive and negative extensions of a~formula as follows:
\begin{align*}
|\phi|^+\coloneqq\{w\in W\mid w\vDash^+\phi\}&&|\phi|^-\coloneqq\{w\in W\mid w\vDash^-\phi\}.
\end{align*}
We say that a~sequent $\phi\vdash\chi$ is \emph{valid on $\mathfrak{M}=\langle W,v^+,v^-\rangle$} (denoted, $\mathfrak{M}\models[\phi\vdash\chi]$) iff $|\phi|^+\subseteq|\chi|^+$ and $|\chi|^-\subseteq|\phi|^-$. A sequent $\phi\vdash\chi$ is \emph{$\BD$-valid} ($\phi\!\models_\BD\!\chi$) iff it is valid on every model. In this case, we will say that $\phi$ \emph{entails}~$\chi$.
\end{definition}

Now, we can use the above semantics to define probabilities on the models. We adapt the definitions from~\cite{KleinMajerRad2021}.
\begin{definition}[$\BD$ models with $\pm$-probabilities]\label{def:measuredmodel}
A \emph{$\BD$ model with a $\pm$-probability} is a tuple $\mathfrak{M}_\mu=\langle\mathfrak{M},\mu\rangle$ with $\mathfrak{M}$ being a $\BD$ model and $\mu:2^W\rightarrow[0,1]$ satisfying:
\begin{description}[noitemsep,topsep=2pt]
\item[$\mathbf{mon}$:] if $X\subseteq Y$, then $\mu(X)\leq\mu(Y)$;
\item[$\mathbf{neg}$:] $\mu(|\phi|^-)=\mu(|\neg\phi|^+)$;
\item[$\mathbf{ex}$:] $\mu(|\phi\vee\chi|^+)=\mu(|\phi|^+)+\mu(|\chi|^+)-\mu(|\phi\wedge\chi|^+)$.
\end{description}
\end{definition}
To facilitate the presentation of the four-valued probabilities defined over $\BD$ models, we introduce additional extensions of $\phi$ defined via $|\phi|^+$ and $|\phi|^-$.
\begin{convention}
Let $\mathfrak{M}=\langle W,v^+,v^-\rangle$ be a $\BD$ model, $\phi\in\LBD$. We set
\begin{align*}
|\phi|^\purebel=&|\phi|^+\setminus|\phi|^- & |\phi|^\puredisbel=&|\phi|^-\setminus|\phi|^+\\
|\phi|^\confl=&|\phi|^+\cap|\phi|^-& |\phi|^\uncert=&W\setminus(|\phi|^+\cup|\phi|^-)
\end{align*}
We call these extensions, respectively, \emph{pure belief}, \emph{pure disbelief}, \emph{conflict}, and \emph{uncertainty in $\phi$}, following~\cite{KleinMajerRad2021}.
\end{convention}
\begin{definition}[$\BD$ models with $\four$-probabilities]\label{def:4model}
A \emph{$\BD$ model with a $\four$-probability} is a tuple $\mathfrak{M}_\four=\langle\mathfrak{M},\mu_\four\rangle$ with $\mathfrak{M}$ being a $\BD$ model and $\mu_\four:2^W\rightarrow[0,1]$ satisfying:
\begin{description}[noitemsep,topsep=2pt]
\item[$\mathbf{part}$:] $\mu_\four(|\phi|^\purebel)+\mu_\four(|\phi|^\puredisbel)+\mu_\four(|\phi|^\uncert)+\mu_\four(|\phi|^\confl)=1$;
\item[$\mathbf{neg}$:] $\mu_\four(|\neg\phi|^\purebel)=\mu_\four(|\phi|^\puredisbel)$, $\mu_\four(|\neg\phi|^\confl)=\mu_\four(|\phi|^\confl)$;
\item[$\mathbf{contr}$:] $\mu_\four(|\phi\wedge\neg\phi|^\purebel)=0$, $\mu_\four(|\phi\wedge\neg\phi|^\confl)=\mu_\four(|\phi|^\confl)$;
\item[$\mathbf{BCmon}$:] if $\mathfrak{M}\models[\phi\vdash\chi]$, then $\mu_\four(|\phi|^\purebel)+\mu_\four(|\phi|^\confl)\leq\mu_\four(|\psi|^\purebel)+\mu_\four(|\psi|^\confl)$;
\item[$\mathbf{BCex}$:] $\mu_\four(|\phi|^\purebel)+\mu_\four(|\phi|^\confl)+\mu_\four(|\psi|^\purebel)+\mu_\four(|\psi|^\confl)=\mu_\four(|\phi\wedge\psi|^\purebel)+\mu_\four(|\phi\wedge\psi|^\confl)+\mu_\four(|\phi\vee\psi|^\purebel)+\mu_\four(|\phi\vee\psi|^\confl)$.
\end{description}
\end{definition}
\begin{convention}
We will further utilise the following naming convention:
\begin{itemize}[noitemsep,topsep=2pt]
\item we use the term \emph{‘$\pm$-probability’} to stand for $\mu$ from Definition~\ref{def:measuredmodel};
\item we call $\mu_\four$ from Defintion~\ref{def:4model} a~\emph{‘$\four$-probability’} or a~\emph{‘four-valued probability’}.
\end{itemize}
Recall that $\pm$-probabilities are referred to as ‘non-standard’ in~\cite{KleinMajerRad2021} and~\cite{BilkovaFrittellaKozhemiachenkoMajerNazari2022arxiv}. As this term is too broad (four-valued probabilities are not ‘standard’ either), we use a different designation.
\end{convention}

Let us quickly discuss the measures defined above. First, observe that $\mu(|\phi|^+)$ and $\mu(|\phi|^-)$ are independent from one another. Thus, $\mu$ gives two measures to each $\phi$, as desired. Second, recall~\cite[Theorems~2--3]{KleinMajerRad2021} that every $\four$-probability on a $\BD$ model induces a $\pm$-probability and vice versa. In the following sections, we will define two-layered logics for $\BD$ models with $\pm$- and $\four$-probabilities and show that they can be faithfully embedded into each other.

\begin{remark}\label{rem:classicality}
Note, that for every $\BD$ model with a $\pm$-pro\-ba\-bi\-li\-ty $\langle W,v^+,v^-,\mu\rangle$ (resp., $\BD$ model with $\four$-probability $\langle W,v^+,v^-,\mu_\four\rangle$), there exist a $\BD$ model $\langle W',v'^+,v'^-,\pi\rangle$ with a~\emph{classical} probability measure $\pi$ s.t.\ $\pi(|\phi|^+)=\mu(|\phi|^+)$ (resp., $\pi(|\phi|^\mathsf{x})=\mu_\four(|\phi|^\mathsf{x})$ for $\mathsf{x}\in\{\purebel,\puredisbel,\confl,\uncert\}$)~\cite[Theorems~4--5]{KleinMajerRad2021}. Thus, we can further assume w.l.o.g.\ that $\mu$ and $\mu_\four$ are \emph{classical probability measures} on $W$.
\end{remark}
\section{Logics for paraconsistent probabilities\label{sec:logics}}
In this section, we provide logics that are (weakly) complete w.r.t.\ $\BD$ models with $\pm$- and $\four$-probabilities. Since conditions on measures contain arithmetic operations on $[0,1]$, we choose an expansion of \L{}ukasiewicz logic, namely, \L{}ukasiewicz logic with~$\triangle$ ($\Luk_\triangle$), for the outer layer. Furthermore, $\pm$-probabilities work with both positive and negative extensions of formulas, whence it seems reasonable to use $\Luk^2$ --- a paraconsistent expansion of $\Luk$ (cf.~\cite{BilkovaFrittellaMajerNazari2020,BilkovaFrittellaKozhemiachenko2021} for details) with two valuations --- $v_1$ (support of truth) and $v_2$ (support of falsity) --- on $[0,1]$. This was done in~\cite{BilkovaFrittellaKozhemiachenkoMajerNazari2022arxiv} --- the resulting logic $\LukProbsquare$ was proven to be complete w.r.t.\ $\BD$ models with $\pm$-probabilities.

We begin by recalling the language and standard semantics of \L{}ukasiewicz logic with~$\triangle$ and its paraconsistent expansion $\Luk^2_\triangle$.
\begin{definition}\label{def:Lukalgebra}
The standard $\Luk_\triangle$-algebra is a tuple $\langle[0,1],{\sim_\Luk},\triangle_\Luk,\wedge_\Luk,\vee_\Luk,\rightarrow_\Luk,\odot_\Luk,\oplus_\Luk,\ominus_\Luk\rangle$ with the operations are defined as follows.
\begin{align*}
{\sim_\Luk}a&\coloneqq1-a&\triangle_\Luk a&\coloneqq\begin{cases}1&\text{if }a=1\\0&\text{otherwise}\end{cases}
\end{align*}
\begin{align*}
a\!\wedge_\Luk\!b&\coloneqq\min(a,b)&a\!\vee_\Luk\!b&\coloneqq\max(a,b)&a\!\rightarrow_\Luk\!b&\coloneqq\min(1,1\!-\!a\!+\!b)\\
a\!\odot_\Luk\!b&\coloneqq\max(0,a\!+\!b\!-\!1)&a\!\oplus_\Luk\!b&\coloneqq\min(1,a\!+\!b)&a\!\ominus_\Luk\!b&\coloneqq\max(0,a\!-\!b)
\end{align*}
\end{definition}
\begin{definition}[\L{}ukasiewicz logic with $\triangle$]\label{def:Lukasiewicz}
The language of $\Luk_\triangle$ is given via the following grammar
\[\LLuk\!\ni\!\phi\coloneqq p\!\in\!\Prop\mid{\sim}\phi\mid\triangle\phi\mid(\phi\wedge\phi)\mid(\phi\vee\phi)\mid(\phi\rightarrow\phi)\mid(\phi\odot\phi)\mid(\phi\oplus\phi)\mid(\phi\ominus\phi)\]
We will also write $\phi\leftrightarrow\chi$ as a shorthand for $(\phi\rightarrow\chi)\odot(\chi\rightarrow\phi)$.

A valuation is a map $v\!:\!\Prop\!\rightarrow\![0,1]$ that is extended to the complex formulas as expected: $v(\phi\!\circ\!\chi)\!=\!v(\phi)\!\circ_\Luk\!v(\chi)$. $\phi$ is \emph{$\Luk_\triangle$-valid} iff $v(\phi)=1$ for every~$v$.
\end{definition}
\begin{remark}\label{rem:smalllanguage}
Note that $\triangle$, ${\sim}$, and $\rightarrow$ can be used to define all other connectives as follows.
\begin{align*}
\phi\vee\chi&\coloneqq(\phi\rightarrow\chi)\rightarrow\chi&\phi\wedge\chi&\coloneqq{\sim}({\sim}\phi\vee{\sim}\chi)&\phi\oplus\chi&\coloneqq{\sim}\phi\rightarrow\chi\\
\phi\odot\chi&\coloneqq{\sim}(\phi\rightarrow{\sim}\chi)&\phi\ominus\chi&\coloneqq\phi\odot{\sim}\chi
\end{align*}
\end{remark}
\begin{definition}[$\Luk^2_\triangle$]\label{def:Luk2triangle}
The language is constructed using the following grammar.
\[\LLuksquare\ni\phi\coloneqq p\in\Prop\mid\neg\phi\mid{\sim}\phi\mid\triangle\phi\mid(\phi\rightarrow\phi)\]
The semantics is given by \emph{two} valuations $v_1$ (support of truth) and $v_2$ (support of falsity) $v_1,v_2:\Prop\rightarrow[0,1]$ that are extended as follows.
\begin{align*}
v_1(\neg\phi)&=v_2(\phi)&v_2(\neg\phi)&=v_1(\phi)\\
v_1({\sim}\phi)&={\sim_\Luk}v_1(\phi)&v_2({\sim}\phi)&={\sim_\Luk}v_2(\phi)\\
v_1(\triangle\phi)&=\triangle_\Luk v_1(\phi)&v_2(\triangle\phi)&={\sim_\Luk}\triangle_\Luk{\sim}_\Luk v_2(\phi)\\
v_1(\phi\rightarrow\chi)&=v_1(\phi)\rightarrow_\Luk v_1(\chi)&v_2(\phi\rightarrow\chi)&=v_2(\chi)\ominus_\Luk v_2(\phi)
\end{align*}
We say that $\phi$ is \emph{$\Luk^2_\triangle$-valid} iff for every $v_1$ and $v_2$, it holds that $v_1(\phi)=1$ and $v_2(\phi)=0$.
\end{definition}
\begin{remark}
Again, the remaining connectives can be defined as in Remark~\ref{rem:smalllanguage}. Furthermore, when there is no risk of confusion, we write $v(\phi)=(x,y)$ to designate that $v_1(\phi)=x$ and $v_2(\phi)=y$.
\end{remark}
We are now ready to present the two-layered logics. We begin with $\LukProbsquare$ from~\cite{BilkovaFrittellaKozhemiachenkoMajerNazari2022arxiv}.
\begin{definition}[$\LukProbsquare$: language and semantics]\label{def:PrLuk2}
The language of $\LukProbsquare$ is given by the following grammar
\begin{align*}
\LLukProbsquare\ni\alpha&\coloneqq\Prob\phi\mid{\sim}\alpha\mid\neg\alpha\mid\triangle\alpha\mid(\alpha\rightarrow\alpha)\tag{$\phi\in\LBD$}
\end{align*}
A~$\LukProbsquare$ model is a tuple $\mathbb{M}=\langle\mathfrak{M},\mu,e_1,e_2\rangle$ with $\langle\mathfrak{M},\mu\rangle$ being a $\BD$ model with $\pm$-probability and $e_1,e_2:\LLukProbsquare\rightarrow[0,1]$ s.t.\ $e_1(\Prob\phi)=\mu(|\phi|^+)$, $e_2(\Prob\phi)=\mu(|\phi|^-)$, and the values of complex formulas being computed following Definition~\ref{def:Luk2triangle}. We say that $\alpha$ is \emph{$\LukProbsquare$ valid} iff $e(\alpha)=(1,0)$ in every model.
\end{definition}

\begin{definition}[$\fourLukProb$: language and semantics]\label{def:4ProbLuk}
The language of $\fourLukProb$ is constructed by the following grammar:
\begin{align*}
\LfourLukProb\ni\alpha&\coloneqq\purebelmod\phi\mid\puredisbelmod\phi\mid\conflmod\phi\mid\uncertmod\phi\mid{\sim}\alpha\mid\triangle\alpha\mid(\alpha\rightarrow\alpha)\tag{$\phi\in\LBD$}
\end{align*}
A $\fourLukProb$ model is a tuple $\mathbb{M}=\langle\mathfrak{M},\mu_\four,e\rangle$ with $\langle\mathfrak{M},\mu_\four\rangle$ being a $\BD$ model with $\four$-probability s.t.\ $e(\purebelmod\phi)\!=\!\mu_\four(|\phi|^\purebel)$, $e(\puredisbelmod\phi)\!=\!\mu_\four(|\phi|^\puredisbel)$, $e(\conflmod\phi)\!=\!\mu_\four(|\phi|^\confl)$, $e(\uncertmod\phi)\!=\!\mu_\four(|\phi|^\uncert)$, and the values of complex formulas computed via Definition~\ref{def:Lukasiewicz}. We say that $\alpha$ is \emph{$\fourLukProb$ valid} iff $e(\alpha)=1$ in every model. A set of formulas $\Gamma$ \emph{entails} $\alpha$ ($\Gamma\models_{\fourLukProb}\alpha$) iff there is no $\mathbb{M}$ s.t.\ $e(\gamma)=1$ for every $\gamma\in\Gamma$ but $e(\alpha)\neq1$.
\end{definition}
\begin{convention}
We will further call formulas of the form $\mathsf{X}\phi$ ($\phi\in\LBD$, $\mathsf{X}\in\{\Prob,\purebelmod,\puredisbelmod,\conflmod,\uncertmod\}$) \emph{modal atoms}. We interpret the value of a modal atom as a degree of certainty that the agent has in $\phi$. For example, $e(\Prob p)=(\frac{3}{4},\frac{1}{2})$ means that the agent's certainty in $p$ is $\frac{3}{4}$ and in $\neg p$ is $\frac{1}{2}$. Similarly, $e(\conflmod q)=\frac{1}{3}$ is construed as ‘the agent is conflicted w.r.t.\ $q$ to the degree $\frac{1}{3}$’.
\end{convention}
To make the semantics clearer, we provide the following example.
\begin{example}\label{ex:semantics}
Consider the following $\BD$ model.
\[\xymatrix{w_0:p^\pm,\xcancel{q}&&w_1:p^-,q^-}\]
And let $\mu=\mu_\four$ be defined as follows: $\mu(\{w_0\})=\frac{2}{3}$, $\mu(\{w_1\})=\frac{1}{3}$, $\mu(W)=1$, $\mu(\varnothing)=0$. It is easy to check that $\mu$ satisfies the conditions of Definitions~\ref{def:measuredmodel} and~\ref{def:4model}. Now let $e$ be the \emph{$\Luk^2_\triangle$ valuation} and $e_\four$ the \emph{$\Luk_\triangle$ valuation} induced by $\mu$ and $\mu_\four$, respectively.

Consider two $\BD$ formulas: $p\!\vee\!q$ and $p$. We have $e(\Prob(p\!\vee\!q))\!=\!\left(\frac{2}{3},\frac{1}{3}\right)$ and $e(\Prob p)=\left(\frac{2}{3},1\right)$. In $\fourLukProb$, we have $e_\four(\purebelmod(p\vee q))=\frac{2}{3}$, $e_\four(\puredisbelmod(p\vee q))=\frac{1}{3}$, $e_\four(\conflmod p)=\frac{2}{3}$, $e_\four(\conflmod(p\vee q)),e_\four(\uncertmod(p\vee q))=0$, $e_\mathbf{4}(\purebelmod p),e(\uncertmod p)=0$, $e_\four(\conflmod p)=\frac{2}{3}$, and $e(\puredisbelmod p)=\frac{1}{3}$.
\end{example}
The following property of $\LukProbsquare$ is going to be useful further in the section.
\begin{lemma}\label{lemma:conflation}
Let $\alpha\in\LLukProbsquare$. Then, $\alpha$ is $\LukProbsquare$ valid iff $e_1(\alpha)=1$ in every $\LukProbsquare$ model.
\end{lemma}
\begin{proof}
Let $\mathbb{M}=\langle W,v^+,v^-,\mu,e_1,e_2\rangle$ be a $\LukProbsquare$ model s.t.\ $e_2(\alpha)\neq0$. We construct a model $\mathbb{M}^*=\langle W,(v^*)^+,(v^*)^-,\mu,e^*_1,e^*_2\rangle$ where $e^*_1(\alpha)\neq1$. To do this, we define new $\BD$ valuations $(v^*)^+$ and $(v^*)^-$ on $W$ as follows.
\begin{align*}
w\in v^+(p),w\notin v^-(p)&\text{ then }w\in(v^*)^+(p),w\notin(v^*)^-(p)\\
w\in v^+(p),v^-(p)&\text{ then }w\notin(v^*)^+(p),(v^*)^-(p)\\
w\notin v^+(p),v^-(p)&\text{ then }w\in(v^*)^+(p),(v^*)^-(p)\\
 w\notin v^+(p),w\in v^-(p)&\text{ then }w\notin(v^*)^+(p),w\in(v^*)^-(p)
\end{align*}
It can be easily checked by induction on $\phi\in\LBD$ that
\begin{align*}
|\phi|^+_\mathbb{M}&=W\setminus|\phi|^-_{\mathbb{M}^*}&|\phi|^-_\mathbb{M}&=W\setminus|\phi|^+_{\mathbb{M}^*}
\end{align*}
Now, since we can w.l.o.g.\ assume that $\mu$ is a~(classical) probability measure on~$W$ (recall Remark~\ref{rem:classicality}), we have that
\begin{align*}
e^*(\Prob\phi)=(1-\mu(|\phi|^-),1-\mu(|\phi|^+))=(1-e_2(\Prob\phi),1-e_1(\Prob\phi))
\end{align*}
Observe that if $e(\alpha)=(x,y)$, then $e(\neg{\sim}\alpha)=(1-y,1-x)$. Furthermore, it is straightforward to verify that the following formulas are valid.
\begin{align*}
\neg{\sim}\neg\alpha&\leftrightarrow\neg\neg{\sim}\alpha&\neg{\sim\sim}\alpha&\leftrightarrow{\sim}\neg{\sim}\alpha\\
\neg{\sim}\triangle\alpha&\leftrightarrow\triangle\neg{\sim}\alpha&\neg{\sim}(\alpha\!\rightarrow\!\alpha')&\leftrightarrow\neg{\sim}\alpha\!\rightarrow\!\neg{\sim}\alpha'
\end{align*}
Hence, $e^*(\alpha)=(1-e_2(\alpha),1-e_1(\alpha))$ for every $\alpha\in\LLukProbsquare$. The result follows.
\end{proof}

At first glance, $\fourLukProb$ gives a more fine-grained view on a $\BD$ model than $\LukProbsquare$ since it can evaluate each extension of a given $\phi\in\LBD$, while $\LukProbsquare$ always considers $|\phi|^+$ and $|\phi|^-$ together. In the remainder of the section, we show that the two logics have, in fact, the same expressivity.

One can see from Definition~\ref{def:PrLuk2} that $\neg\Prob\phi\leftrightarrow\Prob\neg\phi$. Furthermore, $\Luk^2$ admits $\neg$ negation normal forms and is a conservative extension of $\Luk$~\cite{BilkovaFrittellaMajerNazari2020,BilkovaFrittellaKozhemiachenko2021}. Thus, it is possible to push all $\neg$'s occurring in $\alpha\in\LLukProbsquare$ to modal atoms. We will use this fact to establish the embeddings of $\LukProbsquare$ and $\fourLukProb$ into one another.
\begin{definition}\label{def:positiveNNFs}
Let $\alpha\in\LLukProbsquare$. $\alpha^\neg$ is produced from $\alpha$ by successively applying the following transformations.
\begin{align*}
\neg\Prob\phi&\rightsquigarrow\Prob\neg\phi&\neg\neg\alpha&\rightsquigarrow\alpha&\neg{\sim}\alpha&\rightsquigarrow{\sim}\neg\alpha\\\neg(\alpha\rightarrow\alpha')&\rightsquigarrow{\sim}(\neg\alpha'\rightarrow\neg\alpha)&\neg\triangle\alpha&\rightsquigarrow{\sim}\triangle{\sim}\neg\alpha
\end{align*}
\end{definition}

It is easy to check that $e(\alpha)=e(\alpha^\neg)$ in every $\LukProbsquare$ model.
\begin{definition}\label{def:embeddings}
Let $\alpha\in\LLukProbsquare$ be $\neg$-free, we define $\alpha^\four\in\LfourLukProb$ as follows.
\begin{align*}
(\Prob\phi)^\four&=\purebelmod\phi\oplus\conflmod\phi\\
(\heartsuit\alpha)^\four&=\heartsuit\alpha^\four\tag{$\heartsuit\in\{\triangle,{\sim}\}$}\\
(\alpha\rightarrow\alpha')^\four&=\alpha^\four\rightarrow\alpha'^\four
\end{align*}

Let $\beta\in\LfourLukProb$. We define $\beta^\pm$ as follows.
\begin{align*}
(\purebelmod\phi)^\pm&=\Prob\phi\ominus\Prob(\phi\wedge\neg\phi)\\
(\conflmod\phi)^\pm&=\Prob(\phi\wedge\neg\phi)\\
(\uncertmod\phi)^\pm&={\sim}\Prob(\phi\vee\neg\phi)\\
(\puredisbelmod\phi)^\pm&=\Prob\neg\phi\ominus\Prob(\phi\wedge\neg\phi)\\
(\heartsuit\beta)^\pm&=\heartsuit\beta^\pm\tag{$\heartsuit\in\{\triangle,{\sim}\}$}\\
(\beta\rightarrow\beta')^\pm&=\beta^\pm\rightarrow\beta'^\pm
\end{align*}
\end{definition}
\begin{theorem}\label{theorem:embeddings1}
$\alpha\in\LLukProbsquare$ is $\LukProbsquare$ valid iff $(\alpha^\neg)^\four$ is $\fourLukProb$ valid.
\end{theorem}
\begin{proof}
Let w.l.o.g.\ $\mathbb{M}=\langle W,v^+,v^-,\mu,e_1,e_2\rangle$ be a $\BD$ model with $\pm$-probability where $\mu$ is a \emph{classical} probability measure and let $e(\alpha)=(x,y)$. We show that in the $\BD$ model $\mathbb{M}_\four=\langle W,v^+,v^-,\mu,e_1\rangle$ with \emph{four-probability} $\mu$, $e_1((\alpha^\neg)^\four)=x$. This is sufficient to prove the result. Indeed, by Lemma~\ref{lemma:conflation}, it suffices to verify that $e_1(\alpha)=1$ for every $e_1$, to establish the validity of $\alpha\in\LLukProbsquare$. By Lemma~\ref{lemma:conflation}, this is sufficient to prove the result since there to verify the validity of $\alpha\in\LLukProbsquare$, it suffices to verify whether $e_1(\alpha)=1$ for every $e_1$.

We proceed by induction on $\alpha^\neg$ (recall that $\alpha\leftrightarrow\alpha^\neg$ is $\LukProbsquare$ valid). If $\alpha=\Prob\phi$, then $e_1(\Prob\phi)=\mu(|\phi|^+)=\mu(|\phi|^\purebel\cup|\phi|^\confl)$. But $|\phi|^\purebel$ and $|\phi|^\confl$ are disjoint, whence $\mu(|\phi|^\purebel\cup|\phi|^\confl)=\mu(|\phi|^\purebel)+\mu(|\phi|^\confl)$, and since $\mu(|\phi|^\purebel)+\mu(|\phi|^\confl)\leq1$, we have that $e_1(\purebelmod\phi\oplus\conflmod\phi)=\mu(|\phi|^\purebel)+\mu(|\phi|^\confl)=e_1(\Prob\phi)$, as required.

The induction steps are straightforward since the semantic conditions of support of truth in $\Luk^2_\triangle$ coincide with the semantics of $\Luk_\triangle$ (cf.~Definitions~\ref{def:Luk2triangle} and~\ref{def:Lukasiewicz}).
\end{proof}
\begin{theorem}\label{theorem:embeddings2}
$\beta\in\LfourLukProb$ is $\LfourLukProb$ valid iff $\beta^\pm$ is $\LukProbsquare$ valid.
\end{theorem}
\begin{proof}
Assume w.l.o.g.\ that $\mathbb{M}=\langle W,v^+,v^-,\mu_\four,e\rangle$ is a $\BD$ model with a $\four$-probability where $\mu_\four$ is a classical probability measure and $e(\beta)=x$. We define a $\BD$ model with $\pm$-probability $\mathbb{M}^\pm=\langle W,v^+,v^-,\mu_\four,e_1,e_2\rangle$ and show that $e_1(\beta^\pm)=x$. Again, it is sufficient for us by Lemma~\ref{lemma:conflation}.

We proceed by induction on $\beta$. If $\beta=\purebelmod\phi$, then $e(\purebelmod\phi)=\mu_\four(|\phi|^\purebel)$. Now observe that $\mu_\four(|\phi|^+)=\mu(|\phi|^\purebel\cup|\phi|^\confl)=\mu_\four(|\phi|^\purebel)+\mu_\four(|\phi|^\confl)$ since $|\phi|^\purebel$ and $|\phi|^\confl$ are disjoint. But $\mu_\four(|\phi|^+)\!=\!e_1(\Prob\phi)$ and $\mu_\four(|\phi|^\confl)\!=\!\mu_\four(|\phi\!\wedge\!\neg\phi|^+)$ since $|\phi\!\wedge\!\neg\phi|^+\!=\!|\phi|^\confl$. Thus, $\mu_\four(|\phi|^\purebel)=e_1(\Prob\phi\ominus\Prob(\phi\wedge\neg\phi))$ as required.

Other basis cases of $\conflmod\phi$, $\uncertmod\phi$, and $\puredisbelmod\phi$ can be tackled in a similar manner. The induction steps are straightforward since the support of truth in $\Luk^2_\triangle$ coincides with semantical conditions in $\Luk_\triangle$.
\end{proof}
\section{Hilbert-style axiomatisation of $\fourLukProb$\label{sec:axioms}}
Let us proceed to the axiomatisation of $\fourLukProb$. Since its outer layer expands $\Luk_\triangle$, we will need to encode the conditions on $\mu_\four$ therein. Furthermore, since $\Luk$~(and hence, $\Luk_\triangle$) is not compact~\cite[Remark~3.2.14]{Hajek1998}, our axiomatisation can only be \emph{weakly complete} (i.e., complete w.r.t.\ finite theories).

The axiomatisation will consist of two types of axioms: those that axiomatise~$\Luk_\triangle$ and modal axioms that encode the conditions from Definition~\ref{def:4model}. For the sake of brevity, we will compress the axiomatisation of $\Luk_\triangle$ into one axiom that allows us to use $\Luk_\triangle$ theorems\footnote{A Hilbert-style calculus for $\Luk$ can be found in, e.g.~\cite{MetcalfeOlivettiGabbay2008}, and the axioms for $\triangle$ in~\cite{Baaz1996}. A concise presentation of a Hilbert-style calculus for $\Luk_\triangle$ is also given in~\cite{BilkovaFrittellaKozhemiachenkoMajerNazari2022arxiv}.} without proof.
\begin{definition}[$\HfourLukProb$ --- Hilbert-style calculus for $\fourLukProb$]\label{def:HfourLukProb}
The calculus $\HfourLukProb$ consists of the following axioms and rules.
\begin{description}[topsep=2pt,itemsep=3pt]
\item[$\Luk_\triangle$:] $\Luk_\triangle$ valid formulas instantiated in $\LfourLukProb$.
\item[$\mathsf{equiv}$:] $\mathsf{X}\phi\!\leftrightarrow\!\mathsf{X}\chi$ for every $\phi,\chi\!\in\!\LBD$ s.t.\ $\phi\!\dashv\vdash\!\chi$ is $\BD$-valid and $\mathsf{X}\!\in\!\{\purebelmod,\puredisbelmod,\conflmod,\uncertmod\}$.
\item[$\mathsf{contr}$:] ${\sim}\purebelmod(\phi\wedge\neg\phi)$; $\conflmod\phi\leftrightarrow\conflmod(\phi\wedge\neg\phi)$.
\item[$\mathsf{neg}$:] $\purebelmod\neg\phi\leftrightarrow\puredisbelmod\phi$; $\conflmod\neg\phi\leftrightarrow\conflmod\phi$.
\item[$\mathsf{mon}$:] $(\purebelmod\phi\oplus\conflmod\phi)\rightarrow(\purebelmod\chi\oplus\conflmod\chi)$ for every $\phi,\chi\in\LBD$ s.t.\ $\phi\vdash\chi$ is $\BD$-valid.
\item[$\mathsf{part1}$:] $\purebelmod\phi\oplus\puredisbelmod\phi\oplus\conflmod\phi\oplus\uncertmod\phi$.
\item[$\mathsf{part2}$:] $((\mathsf{X}_1\phi\oplus\mathsf{X}_2\phi\oplus\mathsf{X}_3\phi\oplus\mathsf{X}_4\phi)\ominus\mathsf{X}_4\phi)\leftrightarrow(\mathsf{X}_1\phi\oplus\mathsf{X}_2\phi\oplus\mathsf{X}_3\phi)$ with $\mathsf{X}_i\neq\mathsf{X}_j$, $\mathsf{X}_i\in\{\purebelmod,\puredisbelmod,\conflmod,\uncertmod\}$.
\item[$\mathsf{ex}$:] 
$(\purebelmod(\phi\vee\chi)\oplus\conflmod(\phi\vee\chi))\leftrightarrow((\purebelmod\phi\oplus\conflmod\phi)\ominus(\purebelmod(\phi\wedge\chi)\oplus\conflmod(\phi\wedge\chi))\oplus(\purebelmod\chi\oplus\conflmod\chi))$.
\item[MP:] $\dfrac{\alpha\quad\alpha\rightarrow\alpha'}{\alpha'}$.
\item[$\triangle\mathsf{nec}$:] $\dfrac{\HfourLukProb\vdash\alpha}{\HfourLukProb\vdash\triangle\alpha}$.
\end{description}
\end{definition}

The axioms above are simple translations of properties from Definition~\ref{def:4model}. We split \textbf{part} in two axioms to ensure that the values of $\purebelmod\phi$, $\puredisbelmod\phi$, $\conflmod\phi$, and $\uncertmod\phi$ sum up exactly to $1$.
\begin{theorem}\label{theorem:completeness}
Let $\Xi\subseteq\LfourLukProb$ be finite. Then  $\Xi\models_{\fourLukProb}\alpha$ iff $\Xi\vdash_{\HfourLukProb}\alpha$.
\end{theorem}
\begin{proof}
Soundness can be established by the routine check of the axioms' validity. Thus, we prove completeness. We reason by contraposition. Assume that $\Xi\nvdash_{\HfourLukProb}\alpha$. Now, observe that $\HfourLukProb$ proofs are, actually, $\Luk_\triangle$ proofs with additional probabilistic axioms. 
Let $\Xi^*$ stand for $\Xi$ extended with probabilistic axioms built over all pairwise non-equivalent $\LBD$ formulas constructed from $\Prop[\Xi\cup\{\alpha\}]$. 
Clearly, $\Xi^*\nvdash_{\HfourLukProb}\alpha$ either. Moreover, $\Xi^*$ is finite as well since $\BD$ is tabular (and whence, there exist only finitely many pairwise non-equivalent $\LBD$ formulas over a finite set of variables). Now, by the weak completeness of $\Luk_\triangle$, there exists an $\Luk_\triangle$ valuation $e$ s.t.\ $e[\Xi^*]=1$ and $e(\alpha)\neq1$.

It remains to construct a $\fourLukProb$ model $\mathbb{M}$ falsifying $\Xi^*\models_{\fourLukProb}\alpha$ using~$e$. We proceed as follows. First, we set $W=2^{\Lit[\Xi^*\cup\{\alpha\}]}$, and for every $w\in W$ define $w\in v^+(p)$ iff $p\in w$ and $w\in v^-(p)$ iff $\neg p\in w$. We extend the valuations to $\phi\in\LBD$ in the usual manner. Then, for $\mathsf{X}\phi\in\Sf[\Xi^*\cup\{\alpha\}]$ we set $\mu_\four(|\phi|^\mathsf{x})=e(\mathsf{X}\phi)$ according to modality $\mathsf{X}$.

Observe now that any map from $2^W$ to $[0,1]$ that extends $\mu_\four$ is, in fact, a~$\four$-probability. Indeed, all requirements from Definition~\ref{def:4model} are satisfied as $\Xi^*$ contains all the necessary instances of probabilistic axioms and $e[\Xi^*]=1$.
\end{proof}
\begin{remark}
Observe that we could use a \emph{classical probability measure} in the proof of Theorem~\ref{theorem:completeness} because of~\cite[Theorem~5]{KleinMajerRad2021}.
\end{remark}
\section{Decidability and complexity\label{sec:complexity}}
In the completeness proof, we reduced $\HfourLukProb$ proofs to $\Luk_\triangle$ proofs. We know that validity and finitary entailment of $\Luk_\triangle$ are $\conp$-complete (since $\Luk$ is $\conp$-complete and $\triangle$ has truth-functional semantics).

Likewise, $\LukProbsquare$ proofs are also reducible to $\Luk^2$ proofs (cf.~\cite[Theorem~4.24]{BilkovaFrittellaKozhemiachenkoMajerNazari2022arxiv}) from substitution instances of axioms $\Prob\phi\rightarrow\Prob\chi$ (for $\phi\models_\BD\chi$), $\neg\Prob\phi\leftrightarrow\Prob\neg\phi$, and $\Prob(\phi\vee\chi)\leftrightarrow(\Prob\phi\ominus\Prob(\phi\wedge\chi))\oplus\Prob\chi$. Thus, it is clear that the validity and satisfiability of $\fourLukProb$ and $\LukProbsquare$ are $\conp$-hard and $\np$-hard, respectively.

In this section, we provide a simple decision procedure for $\LukProbsquare$ and $\fourLukProb$ and show that their satisfiability and validity are $\np$- and $\conp$-complete, respectively. Namely, we adapt constraint tableaux for $\Luk^2$ defined in~\cite{BilkovaFrittellaKozhemiachenko2021} and expand them with rules for $\triangle$. We then adapt the $\np$-completeness proof $\mathsf{FP}(\Luk)$ from~\cite{HajekTulipani2001} to establish our result.
\begin{definition}[Constraint tableaux for $\Luk^2_\triangle$ --- $\mathcal{T}\left(\Luk^2_\triangle\right)$]\label{def:L2triangleconstrainttableaux}
Branches contain \emph{labelled formulas} of the form $\phi\leqslant_1i$, $\phi\leqslant_2 i$, $\phi\geqslant_1i$, or $\phi\geqslant_2i$, and \emph{numerical constraints} of the form $i\leq j$ with $i,j \in [0,1]$.

Each branch can be extended by an application of one of the rules below.
\[\scriptsize{\begin{array}{cccc}
\neg\!\leqslant_1\!\dfrac{\neg\phi\leqslant_1i}{\phi\leqslant_2i}
&
\neg\!\leqslant_2\!\dfrac{\neg\phi\leqslant_2i}{\phi\leqslant_1i}
&
\neg\!\geqslant_1\!\dfrac{\neg\phi\geqslant_1i}{\phi\geqslant_2i}
&
\neg\!\geqslant_2\!\dfrac{\neg\phi\geqslant_2i}{\phi\geqslant_1i}\end{array}}\]
\[\scriptsize{\begin{array}{cccc}
{\sim}\!\leqslant_1\!\dfrac{{\sim}\phi\leqslant_1i}{\phi\geqslant_11-i}
&
{\sim}\!\leqslant_2\!\dfrac{{\sim}\phi\leqslant_2i}{\phi\geqslant_21-i}
&
{\sim}\!\geqslant_1\!\dfrac{{\sim}\phi\geqslant_1i}{\phi\leqslant_11-i}
&
{\sim}\!\geqslant_2\!\dfrac{{\sim}\phi\geqslant_2i}{\phi\leqslant_21-i}\end{array}}\]
\[\scriptsize{\begin{array}{cccc}
\triangle\!\leqslant_1\!\dfrac{\triangle\phi\geqslant_1i}{i\leq0\left|\begin{matrix}\phi\geqslant_1j\\j\geqslant1\end{matrix}\right.}
&
\triangle\!\geqslant_1\!\dfrac{\triangle\phi\leqslant_1i}{i\geq1\left|\begin{matrix}\phi\leqslant_1j\\j<1\end{matrix}\right.}
&
\triangle\!\leqslant_2\!\dfrac{\triangle\phi\leqslant_2i}{i\geq1\left|\begin{matrix}\phi\leqslant j\\j\leq0\end{matrix}\right.}
&
\triangle\!\geqslant_2\!\dfrac{\triangle\phi\geqslant_2i}{i\leq0\left|\begin{matrix}\phi\geqslant j\\j>0\end{matrix}\right.}\end{array}}\]
\[\scriptsize{\begin{array}{cccc}
\rightarrow\leqslant_1\dfrac{\phi_1\rightarrow\phi_2\leqslant_1i}{i\geq1\left|\begin{matrix}\phi_1\geqslant_11-i+j\\\phi_2\leqslant_1j\\j\leq i\end{matrix}\right.}
&
\rightarrow\leqslant_2\dfrac{\phi_1\rightarrow\phi_2\leqslant_2i}{\begin{matrix}\phi_1\geqslant_2j\\\phi_2\leqslant_2i+j\end{matrix}}
&
\rightarrow\geqslant_1\dfrac{\phi_1\rightarrow\phi_2\geqslant_1 i}{\begin{matrix}\phi_1\leqslant_11-i+j\\\phi_2\geqslant_1j\end{matrix}}
&
\rightarrow\geqslant_2\dfrac{\phi_1\rightarrow\phi_2\geqslant_2i}{i\leq0\left|\begin{matrix}\phi_1\leqslant_2j\\\phi_2\geqslant_2i+j\\j\leq 1-i\end{matrix}\right.}
\end{array}}\]
Let $i$'s be in $[0,1]$ and $x$'s be  variables ranging over the real interval $[0,1]$. We define the translation $\tau$ from labelled formulas to linear inequalities as follows:
$$\tau(\phi\!\leqslant_1\!i)=x_\phi^L\!\leq\!i; \; \tau(\phi\!\geqslant_1\!i)=x_\phi^L\!\geq\!i; \; \tau(\phi\!\leqslant_2\!i)=x_\phi^R\leq i; \; \tau(\phi\!\geqslant_2\!i)=x_\phi^R\!\geq\!i$$
Let $\bullet\in\{\leqslant_1,\geqslant_1\}$ and $\circ\in\{\leqslant_2,\geqslant_2\}$. A tableau branch
$$\mathcal{B}=\{\phi_1\circ i_1,\ldots,\phi_m\circ i_m,\phi'_1\bullet j_1,\ldots,\phi'_n\bullet j_n,k_1\leq l_1,\ldots,k_q\leq l_q\}$$
is \emph{closed} if the system of inequalities
\[\tau(\phi_1\circ i_1),\ldots,\tau(\phi_m\circ i_m),\tau(\phi'_1\bullet j_1),\ldots,\tau(\phi'_n\bullet j_n),k_1\leq l_1,\ldots,k_q\leq l_q\]
does not have solutions. Otherwise, $\mathcal{B}$ is \emph{open}. A tableau is \emph{closed} if all its branches are closed. $\phi$ has a \emph{$\mathcal{T}\left(\Luk^2_\triangle\right)$ proof} if the tableau beginning with $\{\phi\leqslant_1c, c<1\}$ is closed.
\end{definition}

Observe that the $\rightarrow$ and ${\sim}$ rules for $\leqslant_1$ coincide with the analoguous rules in the constraint tableaux for $\Luk$ as given in~\cite{Haehnle1992,Haehnle1994,Haehnle2001HBPL}. Thus, we can use the calculus both for $\fourLukProb$ and $\LukProbsquare$. Note also that we need to build only one tableau for $\LLukProbsquare$ formulas because of Lemma~\ref{lemma:conflation}.

The next statement can be proved in the same manner as~\cite[Theorem~1]{BilkovaFrittellaKozhemiachenko2021}.
\begin{theorem}[Completeness of tableaux]\label{theorem:tableauxcompleteness} ~
\begin{enumerate}[noitemsep,topsep=2pt]
\item $\phi$ is $\Luk_\triangle$ valid iff it has a $\mathcal{T}\left(\Luk^2_\triangle\right)$ proof.
\item $\phi$ is $\Luk^2_\triangle$ valid iff it has a $\mathcal{T}\left(\Luk^2_\triangle\right)$ proof.
\end{enumerate}
\end{theorem}
\begin{theorem}\label{theorem:npcompleteness}
Satisfiability of $\LukProbsquare$ and $\fourLukProb$ is $\np$-complete.
\end{theorem}
\begin{proof}
Recall that $\LukProbsquare$ and $\fourLukProb$ can be linearly embedded into one another (Theorems~\ref{theorem:embeddings1} and~\ref{theorem:embeddings2}). Thus, it remains to provide a non-deterministic polynomial algorithm for one of these logics. We choose $\LukProbsquare$ since it has only one modality.

Let $\alpha\in\LLukProbsquare$. We can w.l.o.g.\ assume that $\neg$ occurs only in modal atoms and that in every modal atom $\Prob\phi_i$, $\phi_i$ is in negation normal form. Define $\alpha^*$ to be the result of the substitution of every $\neg p$ occurring in $\alpha$ with a new variable~$p^*$. It is easy to check that $\alpha$ is satisfiable iff $\alpha^*$ is. We construct a~satisfying valuation for $\alpha^*$.

First, we replace every modal atom $\Prob\phi_i$ with a fresh variable $q_{\phi_i}$. Denote the new formula $(\alpha^*)^-$. It is clear that the size of $(\alpha^*)^-$ ($|(\alpha^*)^-|$) is only linearly greater than $|\alpha|$. We construct a tableau beginning with $\{(\alpha^*)^-\geqslant_1c,c\geq1\}$. This gives us an instance of the MIP equivalent to the $\Luk$-satisfiability of $(\alpha^*)^-$ (cf.~\cite{Haehnle1992,Haehnle1994,Haehnle1999} for more details). Now, write $z_i$ for the values of $q_{\phi_i}$'s in $(\alpha^*)^-$. Our instance of the MIP also has additional variables $x_j$ ranging over $[0,1]$ as well as equalities $k=1$ and $k'=0$ obtained from entries $k\geq1$ and $k'\leq0$. It is clear that both the number of (in)equalities $l_1$ and the number of variables $l_2$ in the MIP are linear w.r.t.\ $|(\alpha^*)^-|$. Denote this instance MIP(1).

We need to show that $z_i$'s are coherent as probabilities of $\phi_i$'s (here, $i\leq n$ indexes the modal atoms of $(\alpha^*)^-$). We introduce $2^n$ variables $u_v$ indexed by $n$-letter words over $\{0,1\}$ and denoting whether the variables of $\phi_i$'s are true under $v^+$.\footnote{Note that $\neg$ does not occur in $(\alpha^*)^-$ and thus we care only about $e_1$ and $v^+$
. Furthermore, while $n$ is the number of $\phi_i$'s, we can add superfluous modal atoms or variables to make it also the number of variables.} We let $a_{i,v}=1$ when $\phi_i$ is true under $v^+$ and $a_{i,v}=0$ otherwise. Now add new equalities denoted with $\mathrm{MIP}(2\exp)$ to MIP(1), namely, $\sum_{v}u_v=1$ and $\sum_{v}(a_{i,v}\cdot u_v)=z_i$. It is clear that the new MIP ($\mathrm{MIP}(1)\cup\mathrm{MIP}(2\exp)$) has a non-negative solution iff $\alpha$ is satisfiable. Furthermore, although there are $l_2+2^n+n$ variables in $\mathrm{MIP}(1)\cup\mathrm{MIP}(2\exp)$, it has no more than $l_1+n+1$ (in)equalities. Thus by~\cite[Lemma~2.5]{FaginHalpernMegiddo1990}, it has a~non-negative solution with at most $l_1+n+1$ non-zero entries. We guess a list $L$ of at most $l_1+n+1$ words $v$ (its size is $n\cdot(l_1+n+1)$). We can now compute the values of $a_{i,v}$'s for $i\leq n$ and $v\in L$ and obtain a new MIP which we denote $\mathrm{MIP}(2\mathrm{poly})$: $\sum_{v\in L}u_v=1$ and $\sum_{v\in L}(a_{i,v}\cdot u_v)=z_i$. It is clear that $\mathrm{MIP}(1)\cup\mathrm{MIP}(2\mathrm{poly})$ is of polynomial size w.r.t.\ $|\alpha|$ and has a non-negative solution iff $\alpha$ is satisfiable. Thus, we can solve it in non-deterministic polynomial time as required.
\end{proof}
\section{Conclusion\label{sec:conclusion}}
We presented logic $\fourLukProb$ formalising four-valued probabilities proposed in~\cite{KleinMajerRad2021} using a two-layered expansion of \L{}ukasiewicz logic with~$\triangle$. We established faithful embeddings between $\fourLukProb$ and $\LukProbsquare$, the logic of $\pm$-probabilities~\cite{BilkovaFrittellaKozhemiachenkoMajerNazari2022arxiv}. Moreover, we constructed a sound and complete axiomatisation of $\fourLukProb$ and proved its decidability using constraint tableaux for $\Luk_\triangle$.

Several questions remain open. 
In~\cite{BilkovaFrittellaKozhemiachenkoMajerNazari2022arxiv}, we presented two-layered logics for reasoning with belief and plausibility functions. These logics employ a~‘two-valued’ interpretation of belief and plausibility (i.e., $\phi$ has two belief assignments: for $\phi$ and for $\neg\phi$). It would be instructive to axiomatise ‘four-valued’ belief and plausibility functions and formalise reasoning with those via a two-layered logic.

Moreover, we have been considering logics whose inner layer lacks implication. It is, however, reasonable to assume that an agent can assign certainty to conditional statements. Furthermore, there are expansions of $\BD$ with truth-functional implications (cf.~\cite{OmoriWansing2017} for examples). A natural next step now is to axiomatise paraconsistent probabilities defined over a logic with an implication.
\bibliographystyle{splncs04}
\bibliography{reference}
\end{document}